\documentclass{article}
\usepackage{amsmath, amsthm, amssymb, amscd}
\usepackage{enumerate}
\usepackage{verbatim}
\usepackage{tikz}
\usepackage{tikz-cd}
\usepackage[noadjust]{cite}

\DeclareMathOperator{\diam}{diam}

\title{Tangents and rectifiability of Ahlfors regular Lipschitz differentiability spaces} 
\author{Guy C. David}
\date{May 10, 2014}
\begin{document}
\maketitle
\newtheorem{thm}{Theorem}[section]
\newtheorem{lemma}[thm]{Lemma}
\newtheorem{prop}[thm]{Proposition}
\newtheorem{cor}[thm]{Corollary}
\newtheorem{claim}[thm]{Claim}

\theoremstyle{remark}
\newtheorem{rmk}[thm]{Remark}

\theoremstyle{definition}
\newtheorem{definition}[thm]{Definition}

\numberwithin{equation}{section}

\newcommand{\obar}[1]{\overline{#1}}
\newcommand{\haus}[1]{\mathcal{H}^n(#1)}
\newcommand{\prob}{\mathbb{P}}
\newcommand{\Tan}{\textnormal{Tan}}
\newcommand{\LIP}{\text{LIP}}
\newcommand{\dist}{\text{dist}}

\begin{abstract}
We study Lipschitz differentiability spaces, a class of metric measure spaces introduced by Cheeger in \cite{Ch99}. We show that if an Ahlfors regular Lipschitz differentiability space has charts of maximal dimension, then, at almost every point, all its tangents are uniformly rectifiable. In particular, at almost every point, such a space admits a tangent that is isometric to a finite-dimensional Banach space. In contrast, we also show that if an Ahlfors regular Lipschitz differentiability space has charts of non-maximal dimension, then these charts are strongly unrectifiable in the sense of Ambrosio-Kirchheim.
\end{abstract}

\section{Introduction}

In 1999, Cheeger \cite{Ch99} introduced a type of metric measure space now known as a ``Lipschitz differentiability space'' (Cheeger did not use this name, which was coined by Bate \cite{Ba13}). These spaces are those permitting a version of Rademacher's theorem: real-valued Lipschitz functions are differentiable almost everywhere with respect to certain ``charts'' covering the space.

Lipschitz differentiability spaces have by now been widely studied; see for example \cite{Ke03}, \cite{Ke03_LD}, and \cite{Ba13}. Well-known examples include Euclidean spaces and Carnot groups. More surprisingly, Cheeger \cite{Ch99} showed that any of the so-called ``PI spaces'' are Lipschitz differentiability spaces; these are doubling metric measure spaces admitting a Poincar\'e inequality in the sense of \cite{HK98}. To find yet more examples, one can observe that any positive-measure subset of a Lipschitz differentiability space is itself a Lipschitz differentiability space (see \cite{Ba13}), though it may certainly fail to admit a Poincar\'e inequality. This indicates that Lipschitz differentiability spaces may have no nice global structure.

The relevant questions for Lipschitz differentiability spaces are therefore local or infinitesimal. In this paper, we will study the Gromov-Hausdorff tangents of these spaces, which describe their infinitesimal behavior. Our main result gives a condition under which such spaces admit Euclidean tangents at almost every point.

\begin{thm}\label{Rntangent}
Let $X$ be a complete, Ahlfors $n$-regular Lipschitz differentiability space containing a chart $U$ of dimension $n$. Then for $\mathcal{H}^n$-almost every point $x\in U$, every tangent of $X$ at $x$ is uniformly rectifiable. In particular, at almost every point of $U$, there is a tangent of $X$ that is bi-Lipschitz equivalent to $\mathbb{R}^n$.

The constants in the uniform rectifiability depend on the point $x$ but not on the particular sequence of scales defining the tangent.
\end{thm}

If one applies Kirchheim's metric differentiation theorem (see \cite{Ki94}, Theorem 2) to this fact, one immediately obtains the following corollary:

\begin{cor}\label{banachtangent}
Let $X$ be a complete, Ahlfors $n$-regular Lipschitz differentiability space containing a chart $U$ of dimension $n$. Then at $\mathcal{H}^n$-almost every point $x\in U$, there is a tangent of $X$ that is isometric to $\mathbb{R}^n$ equipped with a metric induced by a norm.
\end{cor}

\begin{rmk}\label{upperbound}
If $X$ is an Ahlfors $n$-regular Lipschitz differentiability space, then the dimension $k$ of any chart $(U,\phi \colon U\rightarrow\mathbb{R}^k)$ satisfies $k\leq n$ (see Corollary \ref{assouad} below), although this inequality may be strict. Thus, Theorem \ref{Rntangent} and Corollary \ref{banachtangent} are about the case in which the dimension is extremal. 
\end{rmk}

In contrast to Theorem \ref{Rntangent}, one may ask whether a differentiable structure of dimension strictly less than the Ahlfors regularity dimension implies a type of unrectifiability of the space. This is in fact the case.

\begin{thm}\label{strongunrectthm}
Suppose that $s>0$ and that $X$ is an Ahlfors $s$-regular Lipschitz differentiability space containing a chart $U$ of dimension $k$, with $k<s$. Then $U$ is strongly $s$-unrectifiable in the sense of Ambrosio-Kirchheim.
\end{thm}

All the relevant definitions will be given below.

\subsection*{Acknowledgments}
The author wishes to thank Mario Bonk for many helpful discussions. He is also grateful for conversations with Raanan Schul, David Bate, and Enrico Le Donne. The author was partially supported by NSF grant DMS-1162471 and by the Winter 2014 trimester on Random Walks and Asymptotic Geometry of Groups at the Institut Henri Poincar\'e in Paris.

\subsection{Notation and definitions}
We will denote metric spaces by pairs $(X,d)$ and metric measure spaces by triples $(X,d,\mu)$. When the metric (and measure) are understood from context we will denote such a space simply by $X$. If $X$ is a metric space we may also denote the metric on $X$ by $d_X$. Our metric spaces are not necessarily assumed to be complete unless explicitly specified. Our measures $\mu$ will always be Borel regular measures, but they also are not necessarily assumed to be complete measures.

For a real number $\lambda>0$, the rescaled metric space $(X,\lambda d)$ will be written $\lambda X$. A \textit{pointed metric space} is a pair $(X,x)$, where $X$ is a metric space and $x\in X$ is a point, typically called the the base point.  We denote open and closed balls in $X$ by $B(x,r)$ and $\overline{B}(x,r)$, respectively, i.e. we have
$$ B(x,r) = \{y\in X: d(x,y)<r\} $$
and
$$ \overline{B}(x,r) = \{y\in X: d(x,y)\leq r\}. $$

If $E$ is a subset of a metric space $(X,d)$ and $x\in X$, then we define
$$\dist(x, E) = \inf \{d(x,y): y\in E\}.$$

Recall that if $(X,d)$ and $(Y,\rho)$ are metric spaces, then a mapping $f\colon X\rightarrow Y$ is \textit{Lipschitz} if there is a constant $L$ such that
\begin{equation}\label{lipdef}
\rho(f(x),f(x'))\leq L d(x,x')
\end{equation}
for any two points $x,x'\in X$. We denote the infimal value of $L$ such that equation (\ref{lipdef}) holds by $\LIP(f)$. The mapping $f$ is called \textit{bi-Lipschitz} if there is a constant $L\geq 1$ such that
$$ L^{-1}d(x,x') \leq \rho(f(x),f(x'))\leq L d(x,x') $$
for any two points $x,x'\in X$. Two spaces are said to be \textit{bi-Lipschitz equivalent} if there is a bi-Lipschitz map of one onto the other.

A metric measure space $(X,d,\mu)$ is called \textit{doubling} if $\mu$ is a non-trivial Borel regular measure and there exists a constant $C>0$ such that
$$ \mu(B(x,2r)) \leq C\mu(B(x,r)) $$
for all $x\in X$ and $r>0$.

If $\mu$ is a doubling measure on the metric space $(X,d)$, then $(X,d)$ is a \textit{doubling metric space}, which means that there exists a constant $N>0$, depending only on the doubling constant associated to $\mu$, such that every ball of radius $2r$ in $X$ can be covered by at most $N$ balls of radius $r$. A collection of metric spaces is called \textit{uniformly doubling} if every space in the collection is doubling with a uniform upper bound on the constant $N$. See \cite{He01} for more on doubling measures and doubling metric spaces. 

We write $\mathcal{H}^n$ for $n$-dimensional Hausdorff measure. A metric space $(X,d)$ is called \textit{Ahlfors $n$-regular} if there is a constant $C>0$ such that
\begin{equation}\label{AR}
C^{-1} r^n \leq \mathcal{H}^n(\obar{B}(x,r)) \leq  Cr^n
\end{equation}
for all $x\in X$ and $r\leq\diam X$. The two-sided bound (\ref{AR}) automatically implies that the measure $\mathcal{H}^n$ is doubling on $X$. As above, we call a collection of spaces \textit{uniformly Ahlfors $s$-regular} if there are all Ahlfors $s$-regular with the same constant $C$.

\begin{definition}
A \textit{Lipschitz differentiability space} is a metric measure space $(X,d,\mu)$ satisfying the following condition: There are positive measure sets (``charts'') $U_i$ covering $X$, positive integers $n_i$ (the ``dimensions of the charts''), and Lipschitz maps $\phi_i\colon U_i\rightarrow\mathbb{R}^{n_i}$ with respect to which any Lipschitz function is differentiable almost everywhere, in the sense that for $\mu$-almost every $x\in U_i$, there exists a unique $df_x\in\mathbb{R}^{n_i}$ such that
$$ \lim_{y\rightarrow x} \frac{|f(y) - f(x) - df_x\cdot(\phi_i(y)-\phi_i(x))|}{d(x,y)} = 0. $$
Here $df_x\cdot(\phi_i(y)-\phi_i(x))$ denotes the standard scalar product between elements of $\mathbb{R}^{n_i}$.
\end{definition}
In the language of \cite{Ke03}, these are spaces admitting a ``strong measurable differentiable structure''.

A \textit{tangent} of $X$ is a complete metric space that is a pointed Gromov-Hausdorff limit of pointed metric spaces of the form $(X,\lambda_j^{-1} d, x)$, where $\lambda_j\rightarrow 0$. We will say more about tangents and Gromov-Hausdorff convergence in Section \ref{GHspacefunctions} below.

The following definition is due to David\footnote{The Guy David mentioned here and in references \cite{Da88} and \cite{DS97} is a professor at Universit\'e Paris-Sud and has no relation to the author of this paper, who is a graduate student at UCLA. The author wishes to apologize for any confusion generated by this amusing coincidence.} and Semmes and is a strong quantitative version of rectifiability.
\begin{definition}
A metric space is \textit{uniformly rectifiable} (in dimension $n$) if it is Ahlfors $n$-regular and there are constants $\alpha,\beta>0$ such that for every ball $B$ of radius $r$, there is a subset $E\subset B$ with $\haus{E} \geq \alpha \haus{B}$ and a mapping $f\colon E\rightarrow\mathbb{R}^n$ that is $\beta$-bi-Lipschitz.
\end{definition}

Finally, we define the notion of strong unrectifiability used in Theorem \ref{strongunrectthm}. This definition was introduced by Ambrosio and Kirchheim in \cite{AK00}.

\begin{definition}\label{strongunrectdef}
For $s>0$, a metric space $X$ is said to be \textit{strongly $s$-unrectifiable} if
$$ \mathcal{H}^s(f(X)) = 0 $$
for every $N\in\mathbb{N}$ and every Lipschitz map $f\colon X\rightarrow\mathbb{R}^N$,
\end{definition}

By Lemma 5.2 of \cite{AK00}, if $n\in\mathbb{N}$ then any strongly $n$-unrectifiable space is also purely $n$-unrectifiable, in the sense that $H^n(f(E))=0$ whenever $E\subseteq \mathbb{R}^n$ and $f\colon E\rightarrow X$ is Lipschitz. However, the converse is not true. Indeed, there are well-known examples of purely unrectifiable subsets of Euclidean space, but clearly no subset of Euclidean space with positive $\mathcal{H}^s$-measure can be strongly $s$-unrectifiable.

In \cite{AK00}, Theorem 7.4, Ambrosio and Kirchheim construct examples of strongly $s$-unrectifiable metric spaces with positive $\mathcal{H}^s$-measure for every dimension $s>0$. Theorem \ref{strongunrectthm} shows that Ahlfors $s$-regular Lipschitz differentiability spaces with charts of dimension less than $s$ provide more examples of strongly unrectifiable spaces. (Note that, by Remark \ref{upperbound}, any Ahlfors $s$-regular Lipschitz differentiability space with $s\notin \mathbb{Z}$ satisfies the condition automatically.) In addition to all non-abelian Carnot groups, there are now numerous other interesting constructions of such spaces, including those of Bourdon-Pajot \cite{BP99}, Laakso \cite{La00}, and Cheeger-Kleiner \cite{CK13_PI}.

\subsection{Outline of the proof of Theorem \ref{Rntangent}}
Here we give a brief summary of the proof of Theorem \ref{Rntangent}. The starting point is a result of Bate, Theorem \ref{Bate}, that says that in a Lipschitz differentiability space, a generic point of a chart $(U, \phi\colon U\rightarrow\mathbb{R}^n)$ admits $n$ distinct ``broken curves'' through it, along which $\phi$ is differentiable with derivatives pointing in $n$ independent directions.

By modifying an idea of \cite{Le11}  --- that a tangent remains a tangent after a change of base point --- we upgrade this to a special property of the tangents $(Y,y)$ of $(X,x)$. Namely, such a tangent admits a Lipschitz map $G\colon Y\rightarrow\mathbb{R}^n$ (which comes from blowing up $\phi$) such that \textit{every} $z\in Y$ admits $n$ bi-Lipschitz \textit{lines} through it, pointing in ``independent'' directions, on which $G$ is linear. (This is weaker than but similar to Cheeger's notion of a ``generalized linear'' function; see \cite{Ch99}, Section 8.)

By a simple argument, such a map must be a Lipschitz quotient map, i.e., $G(B(x,r))\supseteq B(G(x), cr)$ for some constant $c>0$ and every ball $B(x,r)$. We can then appeal to a theorem of David \cite{Da88}, which implies that such a map from an Ahlfors $n$-regular metric space to $\mathbb{R}^n$ must be bi-Lipschitz on a large subset of every ball. This yields uniform rectifiability of the tangent.

To obtain a single tangent that is bi-Lipschitz equivalent to $\mathbb{R}^n$, we can take a further tangent at a point of density of such a subset. As tangents of tangents are tangents (see \cite{Le11} again), this yields a bi-Lipschitz map from a tangent of $X$ onto $\mathbb{R}^n$.

\begin{rmk}
We do not know, though it is a natural conjecture, whether Theorem \ref{Rntangent} can be strengthened to show that an $n$-dimensional chart $U$ in an Ahlfors $n$-regular Lipschitz differentiability space is itself is $n$-rectifiable. It is possible to show that $U$ is $n$-rectifiable if it admits a bi-Lipschitz embedding into some Euclidean space (see Corollary \ref{embedding} below).
\end{rmk}

We now present the details. In Section \ref{GHspacefunctions} we define the version of Gromov-Hausdorff convergence used in the paper, along with a variant which includes converging Lipschitz functions as well as spaces. In Section \ref{tangentsection} we extend a result of Le Donne about tangent spaces to this setting. Sections \ref{lipdifftangent} and \ref{LQsection} contain the proof that, in doubling Lipschitz differentiability spaces, blow-ups of the coordinate mappings are Lipschitz quotient maps. Sections \ref{recttangent} and \ref{unrect} contain the proofs of Theorem \ref{Rntangent} and Theorem \ref{strongunrectthm}, respectively. Finally, in Section \ref{additionalcors} we present some further corollaries: non-embedding results analogous to those for PI spaces, a sharp dimension bound for differentiable structures, and a rigidity result for Lipschitz differentiability spaces admitting quasi-M\"obius symmetries, in the spirit of Bonk-Kleiner \cite{BK02_rigidity}. 

\section{Gromov-Hausdorff convergence of space-functions}\label{GHspacefunctions}
We will now define Gromov-Hausdorff convergence, first for sequences of metric spaces and then for pairs consisting of a metric space and a Lipschitz function. Our version does not differ materially from that used in, for example, \cite{BBI01} or \cite{KM11}. The following preliminary definition will be useful.

\begin{definition}
A map $\phi \colon (X,d,x)\rightarrow (Y,d',y)$ between pointed metric spaces is called an \textit{$\epsilon$-isometry} if 
\begin{enumerate}[(i)]
\item \label{isom1} For all $a,b \in B_X(x,1/\epsilon)$, we have $|d'(\phi(a),\phi(b)) - d(a,b)| < \epsilon$, and
\item \label{isom2} for all $\epsilon \leq r\leq 1/\epsilon,$ we have $N_\epsilon(\phi(B_X(x,r))) \supseteq B_Y(y,r-\epsilon)$.
\end{enumerate}
Here $N_\epsilon(E)$ denotes the open $\epsilon$-neighborhood of a subset $E$ in a metric space $Y$. Note that we do not ask that $\phi(x)=y$, although it follows from the definition that $d'(\phi(x),y)\leq 2\epsilon$.
\end{definition}

A sequence $\{(X_i, x_i)\}$, $i\in\mathbb{N}$, of pointed metric spaces converges to a metric space $(X,x)$ in the pointed Gromov-Hausdorff sense if for all $\epsilon>0$ there exists $i_0\in\mathbb{N}$ such that, for all $i>i_0$, there are $\epsilon$-isometries
$$\phi_i:(X_i, x_i) \rightarrow (X,x) \text{ and } \psi_i: (X,x)\rightarrow (X_i, x_i).$$
If a sequence of pointed metric spaces is uniformly doubling, then it has a subsequence that converges in the pointed Gromov-Hausdorff sense (see, e.g., \cite{BBI01}, Theorem 8.1.0). This notion of convergence can be associated to a distance function, as we indicate below.

Slightly modifying a definition of Keith \cite{Ke03}, we will call a $(X,x,f)$ a \textit{space-function} if $(X,x)$ is a pointed metric space and $f\colon X\rightarrow\mathbb{R}^n$ is a Lipschitz function, for some $n\in\mathbb{N}$ that will be clear from context. Note that, unlike in \cite{Ke03}, the functions $f$ in our space-functions are always Lipschitz, and they are allowed to map into $\mathbb{R}^n$ rather than $\mathbb{R}$. As an abuse of notation, we will call a space-function ``doubling'', ``complete'', etc. if the underlying space is doubling or complete, and we will call it $L$-Lipschitz if the function $f$ is $L$-Lipschitz.

The notion of Gromov-Hausdorff convergence can be extended to space-functions as, for example, in \cite{Ke03} and \cite{KM11}. We present a version of this here.

\begin{definition}
If $(X,x,f\colon X\rightarrow\mathbb{R}^n)$ and $(Y,y,g\colon Y\rightarrow\mathbb{R}^n)$ are space-functions, we define
\begin{align*}
\tilde{D}((X,d,x,f), (Y,d',y,g)) = \inf \bigg\{&\epsilon>0: \text{there exist } \phi:(X,d,x)\rightarrow (Y,d',y) \text{ and }\\
&\psi:(Y,d',y)\rightarrow (X,d,x)\\
&\text{that are }\epsilon\text{-isometries, and such that}\\
&\sup_{B(x,1/\epsilon)} |f-g\circ\phi|<\epsilon \text{ and } \sup_{B(y,1/\epsilon)}|g-f\circ\psi|<\epsilon\bigg\}
\end{align*}
\end{definition}

\begin{lemma}\label{Dproperties}
If we define $D = \min\{\tilde{D}, 1/2\}$, then $\tilde{D}$ is a ``pseudo-quasi-metric'', by which we mean the following:

\begin{enumerate}[\normalfont (i)]
\item\label{sym} $D$ is finite, non-negative, and symmetric.
\item\label{pseudometric} The $D$-distance between two doubling space-functions $(X,x,f)$ and $(Y,y,g)$ is zero if and only if there is a surjective isometry $i\colon \obar{X}\rightarrow \obar{Y}$ such that $g\circ i = f$, where $g$ and $f$ are identified with their extensions to the completions $\obar{X}$ and $\obar{Y}$.
\item\label{quasimetric} $D$ satisfies the quasi-triangle inequality
$$ D\left((X,x,f), (Z,z,h)\right) \leq 2\left(D\left((X,x,f),(Y,y,g)\right) + D\left((Y,y,g),(Z,z,h)\right)\right).$$
\end{enumerate}
\end{lemma}
\begin{proof}
It is clear from the definition that $\tilde{D}$, and therefore $D$, is finite, non-negative, and symmetric, and so (\ref{sym}) holds.

If $D((X,x,f),(Y,y,g))=0$ then there $1/i$-isometries $\phi_i\colon (X,x)\rightarrow (Y,y)$ such that
$$\sup_{B(y,i)}|g-f\circ\phi_i|<1/i.$$
We can extend $\phi_i$ to a map from $\overline{X}$ to $\overline{Y}$ as a $2/i$-isometry. Because $X$ and $Y$ are doubling, $\overline{X}$ and $\overline{Y}$ are proper: closed balls are compact. Therefore the maps $\phi_i$ sub-converge uniformly on compact sets to an isometry from $X$ to $Y$ satisfying the conditions of the lemma.

Conversely, if such an isometry exists, then it is clear that $D((X,x,f),(Y,y,g))=0$. Therefore (\ref{pseudometric}) holds.

The quasi-triangle inequality (\ref{quasimetric}) for $D$ follows from the fact that 
$$(2(\epsilon+\delta))^{-1} \leq \min\{\epsilon^{-1}-2\delta, \delta^{-1}-2\epsilon\}$$
if $0<\epsilon, \delta<1/2$. Indeed, this is inequality exactly what is needed to show that the composition of an $\epsilon$-isometry and a $\delta$-isometry is a $2(\epsilon+\delta)$-isometry.
\end{proof}

Although the function $D$ is not a metric, the previous lemma says that it is similar enough for our application. We will therefore say that a sequence of space-functions $(X_n, x_n,f_n)$ ``converges in $D$'' to a space-function $(X,x,f)$ if
$$D((X_n, x_n,f_n), (X,x,f))\rightarrow 0 \text{ as } n\rightarrow\infty.$$
The convergence in $D$ of a sequence of space-functions implies that the pointed metric spaces converge in the pointed Gromov-Hausdorff sense (as defined above). Conversely, by a standard Arz\'ela-Ascoli type argument, if $(X_n, x_n, f_n)$ are $C$-doubling, $L$-Lipschitz space functions with $\{f_n(x_n)\}$ bounded, and if $(X_n, d_n, x_n)\rightarrow (X,d,x)$ in the pointed Gromov-Hausdorff sense, then there is a subsequence  $\{(X_{n_k}, x_{n_k}, f_{n_k})\}$ and a Lipschitz function $f\colon X\rightarrow\mathbb{R}$ such that
$$ (X_{n_k}, x_{n_k}, f_{n_k}) \rightarrow (X, x, f) $$
in the metric $D$.

If $(X,x)$ is a pointed metric space, and $f\colon X\rightarrow\mathbb{R}^n$ is Lipschitz, then we denote by $\Tan(X,x,f)$ the collection of space functions $(Y,y,g)$ such that $Y$ is complete and
$$ \left(\frac{1}{\lambda_i}X, x, \frac{1}{\lambda_i}(f-f(x))\right) \rightarrow (Y,y,g) $$
for some sequence of positive real numbers $\lambda_i$ converging to zero. This is the collection of tangents of $X$ at $x$. If $X$ is doubling and $f$ is Lipschitz, then $\Tan(X,x,f)$ is always non-empty, by the above standard facts about Gromov-Hausdorff convergence. (Recall that a space-function is always at $D$-distance zero from a space-function whose underlying space is complete. Thus, there is no issue with assuming tangents to be complete, and this will be convenient later on.)

\begin{lemma}\label{GHproperties}
The following properties are preserved under Gromov-Hausdorff convergence of a sequence of space functions $\{(X_i,x_i,f_i)\}\rightarrow (X,x,f)$:
\begin{itemize}
\item If the functions $f_i$ are all $L$-Lipschitz, then so is $f$.
\item If the functions $f_i$ are all $L$-bi-Lipschitz, then so is $f$.
\item If the spaces $X_i$ are uniformly doubling metric spaces, then $X$ is doubling.
\item If the spaces $X_i$ are uniformly Ahlfors $n$-regular, then $X$ is Ahlfors $n$-regular.
\end{itemize}
\end{lemma}
\begin{proof}
The first three of these properties are easy to check, and the fourth can be found in, e.g., Lemma 8.29 of \cite{DS97}.
\end{proof}

The following lemma about convergence of space-functions will be useful.

\begin{lemma}\label{halfconvergence}
Suppose that $(X,x,f)$ and $(Y,y,g)$ are Lipschitz space-functions (mapping into Euclidean space of the same dimension). Suppose that $\phi\colon X\rightarrow Y$ is an $\epsilon$-isometry such that $\phi(x)=y$ and
$$\sup_{B(x,1/\epsilon)} |f - g\circ \phi| < \epsilon<1.$$
Then
$$D\left((X,x,f), (Y,y,g)\right) < C\epsilon, $$
where $C$ depends only on the Lipschitz constants of $f$ and $g$.
\end{lemma}
\begin{proof}
For simplicity, we denote the metrics on $X$ and $Y$ both by $d$. Let $N\subset B(x,1/\epsilon)$ be a maximal separated $\epsilon$-net. In other words,
$$ d(y,z) \geq \epsilon $$
if $y,z\in N$ and $y\neq z$, and
$$ \dist(z,N)<\epsilon $$
for all $z\in B(x,1/\epsilon)$. We can also arrange that $x\in N$.

The fact that $\phi$ is an $\epsilon$-isometry implies that $\phi|_N$ is injective. Let $N'=\phi(N)\subset Y$. Because $\phi$ is an $\epsilon$-isometry, we know that every point of $B(y,1/2\epsilon)$ is within $3\epsilon$ of a point in $N'$.

Let $\pi\colon Y\rightarrow N'$ denote any choice of closest-point projection, i.e., $\pi(Y)\subset N'$ and $d(y,\pi(y)) = \dist(y,N')$. Then $\pi$ preserves distances up to an additive error of $6\epsilon$ for points in $B(y,1/2\epsilon)$. Let
$$ \psi = \left(\phi|_N\right)^{-1} \circ \pi : Y \rightarrow X. $$

We first claim that $\psi$ is a $7\epsilon$-isometry. Fix $y_1, y_2\in B(y,1/7\epsilon)$. We have
\begin{align*}
|d(\psi(y_1), \psi(y_2)) - d(y_1, y_2)| &\leq |d(\phi^{-1}(\pi(y_1)), \phi^{-1}(\pi(y_2))) - d(\pi(y_1), \pi(y_2))| + |d(\pi(y_1), \pi(y_2)) - d(y_1, y_2)|\\
&\leq \epsilon + 6\epsilon\\
&= 7\epsilon.
\end{align*}

In addition, for $r\leq 1/(7\epsilon)$,
$$\psi(B(y,r))\supseteq N \cap B(x, r-\epsilon) $$
and therefore
$$ N_{7\epsilon}(\psi(B(y,r))) \supseteq B(x,r-7\epsilon). $$

We now claim that 
$$\sup_{B(y,1/7\epsilon)} |g - f\circ \psi| < C\epsilon,$$
where $C$ depends only on the Lipschitz constant of $g$. For $z\in B(y, 1/7\epsilon)$, we have
\begin{align*}
|g(z) - f(\psi(z))| &= |g(z) - f((\phi|_N)^{-1}(\pi(z)))|\\
&\leq |g(z) - g(\pi(z))| + |g(\pi(z)) - f((\phi|_N)^{-1}(\pi(z)))|\\
&\leq 6\epsilon \LIP(g) + \epsilon.
\end{align*}
This completes the proof.
\end{proof}

At this point, we remark that all spaces in this paper are doubling and therefore separable, so they admit isometric embeddings into the Banach space $\ell^\infty(\mathbb{N})$. Thus, the collection of all doubling metric spaces up to isometry can be identified with a subset of the power set of $\ell^\infty(\mathbb{N})$, and so there are no set-theoretic difficulties with this object.

Though $D$ is not a metric, we nonetheless let the \textit{$D$-diameter} of a collection $\mathcal{C}$ of space functions be
$$ \sup \{ D((X,x,f), (Y,y,g)) : (X,x,f), (Y,y,g) \in \mathcal{C}\}.$$

\begin{lemma}\label{separable}
Let $\mathcal{M}$ be a collection of doubling, $L$-Lipschitz space-functions (mapping into the same $\mathbb{R}^n$). Then for any $\eta>0$, $\mathcal{M}$ is contained in a countable union of sets $B_l$, $l\in \mathbb{N}$, of $D$-diameter at most $\eta$.
\end{lemma}
\begin{proof}
We consider the countable collection of all space-functions $(X,x,f)$ such that
\begin{itemize}
\item $X$ is finite and all distances between points of $X$ are rational, and
\item $f$ takes values in $\mathbb{Q}^n \subset \mathbb{R}^n$.
\end{itemize}

Given $(Y,y,g)\in \mathcal{M}$, we will show that it is within $D$-distance $\eta/4$ of such a space-function. The quasi-triangle inequality for $D$ (Lemma \ref{Dproperties}) then concludes the proof.

Let $\delta>0$ be a small constant to be chosen later, depending only on $\eta$ and $L$. Let $N$ denote a finite maximal $\delta$-net in $B(y, 1/\eta)\subseteq Y$, which we assume contains $y$. The set $N$ is finite because $Y$ is doubling. We consider $N$ as a metric space equipped with the restriction of the metric from $Y$.

By Kuratowski's theorem (\cite{He01}, p.\ 99), $N$ isometrically embeds into $\left(\mathbb{R}^m, ||\cdot||_{\ell^\infty}\right)$ for some $m\in\mathbb{N}$. Here $\left(\mathbb{R}^m, ||\cdot||_{\ell^\infty}\right)$ denotes $\mathbb{R}^m$ equipped with the metric induced by the norm
$$ ||x||_{\ell^\infty} = \max\{|x_i| : i=1,\dots,m\}$$
for $x=(x_1, \dots, x_m)\in\mathbb{R}^m$.

We form a new metric space $(N', d')$ in the following way: For each $a\in N\subset \mathbb{R}^m $, choose $a'\in \mathbb{Q}^m \subset \mathbb{R}^m$ within $\delta/8$ of $a$. Let $(N',d')$ denote the metric space on the set of all these new points $a'$ equipped with the restriction of the $\ell^\infty$ metric from $\mathbb{Q}^m$. Note that all distances in $(N',d')$ are rational.

Let $\psi\colon N'\rightarrow N \subset Y$ be the obvious bijection between points of $N'$ and points of $N$, and let $y' = \psi^{-1}(y)$. It is clear that $\psi$ is an $(\eta/2)$-isometry if $\delta$ is sufficiently small depending on $\eta$. Let $f\colon N'\rightarrow\mathbb{R}$ be defined so that $f(x)$ is a rational number within $\eta/2$ of $g(\psi(x))$. Thus, $g\circ\psi$ is within $\eta/2$ of $f$ by definition.

By Lemma \ref{halfconvergence}, 
$$ D((Y,y,g), (N',y', f))\leq C\delta \leq \eta/2, $$
where $C$ depends only on $L$, and $\delta$ is chosen in addition to be less than $\eta/(2C)$. This proves the lemma.
\end{proof}

\section{Moving the base points of tangents}\label{tangentsection}

This section is devoted to the proof of the following result, which is an extension of a result of Le Donne \cite{Le11}.

\begin{prop}\label{movedtangent}
Suppose $(X,d,\mu)$ is a doubling metric measure space and $f\colon X\rightarrow\mathbb{R}^n$ is Lipschitz. Then, for $\mu$-almost every $x\in X$, for all $(Y,y,g)\in \textnormal{\Tan}(X,x,f)$, and for all $y'\in Y$, we have $(Y, y', g-g(y'))\in \textnormal{\Tan}(X,x,f)$.
\end{prop}

As we have not assumed that the measure $\mu$ is complete, the exceptional set in Proposition \ref{movedtangent} need not be measurable. We define the outer measure $\mu^*$ by
$$ \mu^*(A) = \inf\{\mu(B) : B\text{ Borel, } B\supseteq A\}. $$
Proposition \ref{movedtangent} says that the exceptional set on which the conclusion fails has outer measure zero. Such a set is contained in a Borel set of measure zero.

We say that point $a$ is a \textit{point of outer density} of a set $A$ if $a\in A$ and
$$ \lim_{r\rightarrow 0} \frac{\mu^*(A\cap B(x,r))}{\mu(B(x,r))} = 1.$$

Every subset of $X$ with positive outer measure has a point of outer density. Indeed, for any such set $A\subseteq X$ there exists a Borel set $B\supset A$ with $\mu(B) = \mu^*(A)>0$.  We have that
$$ \mu^*(A \cap E) = \mu(B\cap E) $$
for any Borel set $E\subseteq X$, from which it follows that any point of density of $B$ is a point of outer density of $A$.

\begin{lemma}\label{densitytangent}
Let $(X,d,\mu)$ be a doubling metric measure space, $f\colon X\rightarrow\mathbb{R}^n$ be Lipschitz, and let $A\subset X$ be a subset with a point of outer density at $a\in A$. Then $\textnormal{\Tan}(A,a,f) = \textnormal{\Tan}(X,a,f)$.
\end{lemma}
\begin{proof}
The proof of this is an easy modification of the proof of Proposition 3.1 in \cite{Le11}, which we omit. 
\end{proof}

\begin{proof}[Proof of Proposition \ref{movedtangent}]
We closely follow the argument in \cite{Le11}. Our goal is to show that the set
$$ \left\{x\in X: \text{ there exists } (Y,y,g)\in \Tan(X,x,f) \text{ and } y'\in Y \text{ such that } (Y,y',g-g(y'))\notin \Tan(X,x,f)\right\} $$
has outer measure zero.

Consider the collection $\mathcal{M}$ consisting of $(X,x,f)$ and all its rescalings and tangents. Note that $\mathcal{M}$ is a collection of uniformly doubling, uniformly Lipschitz space-functions. Using Lemma \ref{separable}, we see that for each $k\in \mathbb{N}$, there exist countably many collections $B_l$ of space-functions such that, for all $l$,
$$\text{diam}_D(B_l) < 1/4k$$
and $\mathcal{M} \subseteq \cup B_l$. 

It therefore suffices to show that, for all $k,l,m \in \mathbb{N}$, the set
\begin{align*}
\bigg\{x\in X: &\text{ there exists } (Y,y,g)\in \Tan(X,x,f) \text{ and } y'\in Y \text{ such that } \\
& (Y, y',g-g(y')) \in B_l \text{ and } D\left(\left(Y,y',g-g(y')\right), \left(\frac{1}{t}X, x, \frac{1}{t}(f-f(x))\right)\right)>\frac{1}{k} \text{ for all } t\in\left(0,1/m \right) \bigg\}
\end{align*}
has outer measure zero.

Suppose that, for some $k,l,m\in\mathbb{N}$, the set above has positive outer measure, and call it $A\subseteq X$. Let $a$ be a point of outer density of $A$. Then there exists $(Y,y,g)\in\Tan(X,a,f)$ and $y'\in Y$ such that
$$ (Y,y', g-g(y'))\in B_l $$
and
$$ D\left( \left(Y,y',g-g(y')\right), \left(\frac{1}{t}X, a, \frac{1}{t}(f-f(x))\right)\right) >\frac{1}{k},$$
for all $t\in\left(0,1/m \right).$

Because $(Y,y,g)\in\Tan(X,a,f)=\Tan(A,a,f)$, there are sequences $\lambda_n\rightarrow 0$ and $\epsilon_n\rightarrow 0$, as well as $\epsilon_n$-isometries $\phi_n\colon  (Y,y)\rightarrow (\frac{1}{\lambda_n}X,a)$ taking values in $A$ and satisfying
$$ \sup_{B(y, \epsilon_n^{-1})} \left|g - \frac{1}{\lambda_n}(f\circ \phi_n-f(x)) \right| \leq \epsilon_n.$$

Let $a_n = \phi_n(y')\in A\subseteq X$. Note that
\begin{equation}\label{andistance}
d_X(a_n, a) = O(\lambda_n) \rightarrow 0
\end{equation}
as $n\rightarrow\infty$.

Consider the space-functions
$$ \left(\frac{1}{\lambda_n}X, a_n, \frac{1}{\lambda_n}(f-f(a_n))\right). $$
We now make the following claim:

\begin{claim}\label{claim}
In the distance $D$, we have the convergence
$$ \left(\frac{1}{\lambda_n}X, a_n, \frac{1}{\lambda_n}(f-f(a_n))\right) \rightarrow (Y,y',g-g(y')) $$
\end{claim}
\begin{proof}[Proof of Claim \ref{claim}]
Consider the same mappings $\phi_n$ as before, now considered as mappings
$$\phi_n: (Y,y') \rightarrow \left(\frac{1}{\lambda_n}X, a_n\right). $$
We will first show that if $n$ is sufficiently large, $\phi_n$ is a $2\epsilon_n$-isometry with these base points. By (\ref{andistance}), if $n$ is sufficiently large, then
$$ B_Y(y', (2\epsilon_n)^{-1}) \subset B_Y(y,\epsilon_n^{-1})$$
and so $\phi_n$ satisfies property (\ref{isom1}) of a $2\epsilon_n$-isometry.

In addition, if $r\leq (2\epsilon_n)^{-1}$ and $n$ is sufficiently large, then
$$ B_{\lambda_n^{-1} X}(a_n,r-2\epsilon_n) \subset B_{\lambda_n^{-1} X}(x,1/\epsilon_n-\epsilon_n).$$
Therefore, if $z\in B_{\lambda_n^{-1} X}(a_n,r-2\epsilon_n)$ then $z$ is within $\lambda_n^{-1}X$-distance $\epsilon_n$ of a point $\phi_n(w)$, where $w\in  B_Y(y',1/\epsilon_n)$. A simple application of the triangle inequality and the properties of $\phi_n$ shows that $w$ must be in $B_Y(y',r)$. Therefore,
$$ B_{\lambda_n^{-1} X}(a_n,r-2\epsilon_n) \subset N_{2\epsilon_n}\left( B_Y(\phi_n(y'),r)\right)$$
which verifies property (\ref{isom2}) of a $2\epsilon_n$-isometry.

Thus, for $n$ large, each mapping $\phi_n$ is a $2\epsilon_n$-isometry from $(Y,y')$ to $(\frac{1}{\lambda_n}X, a_n)$. In addition, we have, for $z\in  B(y', (2\epsilon_n)^{-1})$,
\begin{align*}
\left|(g(z)-g(y')) - \frac{1}{\lambda_n}(f(\phi_n(z)) - f(a_n))\right| &\leq \left|g(z) - \frac{1}{\lambda_n}(f(\phi_n(z)) - f(a))\right|\\
&\qquad +\left|g(y') - \frac{1}{\lambda_n}(f(a_n) - f(a))\right|\\
&= \left|g(z) - \frac{1}{\lambda_n}(f(\phi_n(z)) - f(a))\right|\\
&\qquad +\left|g(y') - \frac{1}{\lambda_n}(f(\phi_n(y')) - f(a))\right|\\
&\leq \epsilon_n + \epsilon_n\\
&= 2\epsilon_n
\end{align*}
Thus, the mappings $\phi_n\colon  (Y,y') \rightarrow (\frac{1}{\lambda_n}X, a_n)$ each satisfy the conditions of Lemma \ref{halfconvergence}, and so we see that, for some $C>0$ independent of $n$,
$$ D\left( \left(Y,y', g-g(y')\right),  \left(\frac{1}{\lambda_n}X, a_n, \frac{1}{\lambda_n}(f-f(a_n))\right)\right) \leq C\epsilon_n \rightarrow 0.$$
\end{proof}

Therefore, for $n$ sufficiently large, we have
\begin{equation}\label{close}
 D\left(\left(\frac{1}{\lambda_n}X, a_n, \frac{1}{\lambda_n}(f-f(a_n))\right), \left(Y,y',g-g(y')\right)\right) < \frac{1}{4k}.
\end{equation}

Now, since $a_n\in A$, there are space-functions $(Y_n, y_n, g_n)\in \Tan(X,a_n,f)$ and points $y'_n\in Y_n$ such that
$$ (Y_n, y'_n, g_n - g_n(y'_n)) \in B_l, $$
and
$$ D\left(\left(\frac{1}{t}X, a_n, \frac{1}{t}(f-f(a_n))\right),\left(Y_n, y'_n, g_n - g_n(y'_n)\right)\right) > 1/k,$$
for all $t\in (0,1/m)$.

We then have, for $n$ large,
\begin{align*}
\frac{1}{k} &< D\left( \left(Y_n, y'_n, g_n - g_n(y'_n)\right) , \left(\frac{1}{\lambda_n}X, a_n, \frac{1}{\lambda_n}(f-f(a_n))\right) \right)\\
&\leq 2\bigg(D\left( \left(Y_n, y'_n, g_n - g_n(y'_n)\right) , \left(Y, y', g-g(y')\right)\right)\\
&\qquad + D\left( \left(Y, y', g - g(y')\right) , \left(\frac{1}{\lambda_n}X, a_n, \frac{1}{\lambda_n}(f-f(a_n))\right) \right)\bigg)\\
&< 2\left(\frac{1}{4k}+\frac{1}{4k}\right),
\end{align*}
where the first $\frac{1}{4k}$ term arises because both spaces are in $B_l$ and the second comes from (\ref{close}). This is a contradiction.

\end{proof}

\section{Relationship to Lipschitz differentiability}\label{lipdifftangent}

We now investigate Lipschitz differentiability spaces. From now on, all metric measure spaces are assumed to be doubling and complete (but not necessarily Ahlfors regular until the proof of Theorem \ref{Rntangent}).

The following notation, borrowed from \cite{Ba13}, is useful: Let $\Gamma(X)$ be the collection of all bi-Lipschitz mappings
$$ \gamma:D_\gamma \rightarrow X$$
where $D_\gamma \subset \mathbb{R}$ is a compact set containing $0$.

We use the following result of Bate \cite{Ba13}, which is a consequence of his investigations of Alberti representations in Lipschitz differentiability spaces.

\begin{thm}[\cite{Ba13}, Corollary 6.7]\label{Bate}
Let $(U,\phi)$ be an $n$-dimensional chart in a complete Lipschitz differentiability space $(X,d,\mu)$. Then for almost every $x\in U$, there exist $\gamma^x_1, \dots, \gamma^x_n\in \Gamma(X)$ such that each $\gamma^x_i(0)=x$, $0$ is a density point of $(\gamma^x_i)^{-1}(U)$, and the derivatives $(\phi\circ\gamma^x_i)'(0)$ exist and are linearly independent.
\end{thm}

The property given in the conclusion of Theorem \ref{Bate} admits an improvement if one passes to tangents. An \textit{$L$-bi-Lipschitz line} in a metric space $X$ is an $L$-bi-Lipschitz map $l\colon \mathbb{R}\rightarrow X$.

\begin{prop}\label{lines}
Let $(X,d,\mu)$ be a complete doubling metric measure space and let $f \colon X\rightarrow\mathbb{R}^n$ be a Lipschitz function. Suppose that there is a set $A$ of positive measure such that for every $x\in A$, there exists $\gamma^x\in \Gamma(X)$ with $\gamma^x(0)=x$, $0$ a density point of $D_{\gamma^x}$, and such that $v_x=(f\circ\gamma^x)'(0)$ exists and is non-zero.

Then for almost every $x\in A$, every $(Y,y,g)\in \Tan(X,x,f)$ has the following property:

There is $L\geq 1$ such that for every $z\in Y$, there exists an $L$-bi-Lipschitz line $l$ with $l(0) = z$ that satisfies
$$ g(l(t)) = g(z) + tv_x$$
for all $t\in \mathbb{R}$.

The constant $L$ depends on the point $x$ but not on the sequence of scales defining the tangent.
\end{prop}
\begin{proof}
Because the conclusion is supposed to hold for almost every $x\in A$, we may assume that $x$ is among the full-measure set of points for which the conclusion of Proposition \ref{movedtangent} holds.

Consider any $(Y,y,g)\in \Tan(X,x,f)$. There is a sequence $\{\lambda_n\}$ tending to zero such that
$$ \left(\lambda_n^{-1} X, x, \lambda_n^{-1}(f-f(x))\right) \rightarrow (Y, y, g). $$
Fix $\epsilon_n$-isometries $\phi_n \colon (Y,y)\rightarrow (\lambda_n^{-1} X, x)$ and $\psi_n \colon (\lambda_n^{-1}X, x)\rightarrow  (Y,y)$ such that
$$ \sup_{B(x,1/\epsilon_n)} |\lambda_n^{-1}(f-f(x))-g\circ\phi|<\epsilon_n \text{ and } \sup_{B(y,1/\epsilon)}|g-\lambda_n^{-1}(f\circ\psi-f(x))|<\epsilon_n, $$
where $\epsilon_n\rightarrow 0$ as $n\rightarrow\infty$. 

We first claim the following: there is an $L$-bi-Lipschitz line $l\colon \mathbb{R}\rightarrow Y$ such that $l(0)=y$ and $g(l(t)) = tv_x$ for all $t\in\mathbb{R}$. In other words, we first claim that the conclusion of the proposition holds when $z$ is actually the base point $y$ of the tangent.

To find the line $l$, we blow up the curve $\gamma^x$ at $t=0\in\mathbb{R}^n$, along the same sequence of scales $\{\lambda_n\}$. Although we have not defined the Gromov-Hausdorff convergence of functions mapping into metric spaces other than $\mathbb{R}^n$, for this one can use the theory developed in \cite{DS97}, Chapter 8.

Passing to a subsequence and using again standard facts about blowups at points of density (see \cite{DS97}, Lemmas 9.12 and 9.13), this gives a bi-Lipschitz line $l$ in $Y$ through $y$. 

By \cite{DS97}, Lemma 8.19, this line $l$ has the following property: There are maps $\sigma_n\colon \mathbb{R}\rightarrow D_\gamma$ such that
$$ \lim_{n\rightarrow \infty} \lambda_n^{-1} \left|\sigma_n(t) - \lambda_n t\right| = 0 $$
and
$$ l(t) = \lim_{n\rightarrow \infty} \psi_n (\gamma(\sigma_n(t))) $$
uniformly in $t$ on bounded subsets of $\mathbb{R}$.

Recall also that $g$ is given by the limit
$$ g(z) = \lim_{n\rightarrow\infty} \frac{1}{\lambda_n}\left(f(\phi_n(z)) - f(x)\right) $$
uniformly on bounded subsets of $Y$.

Therefore, using the uniformity of the convergence and the Lipschitz property of $f$ and $\gamma$, we have that
\begin{align*}
g(l(t)) &= \lim_{n\rightarrow\infty} \frac{1}{\lambda_n}\left(f(\phi_n(l(t))) - f(x)\right)\\
&= \lim_{n\rightarrow\infty} \frac{1}{\lambda_n} \left( f(\phi_n(\psi_n(\gamma^x(\sigma_n(t))))) - f(x)\right)\\
&= \lim_{n\rightarrow\infty} \frac{1}{\lambda_n}\left( f(\gamma^x(\sigma_n(t))) - f(x)\right)\\
&= \lim_{n\rightarrow\infty} \frac{1}{\lambda_n}\left( f(\gamma^x(\lambda_n t)) - f(x)\right)\\
&= t(f\circ \gamma^x)'(0)\\
&= tv_x
\end{align*}
for all $t\in \mathbb{R}$. Thus, we see that $(g\circ l)(t)=tv_x$.

This gives the conclusion of the proposition at the base point $y\in Y$. Now consider any point $z\in Y$. By Proposition \ref{movedtangent}, $(Y,z,g-g(z))\in \Tan(X,x,f)$. Therefore, by the preceding argument, we get the conclusion of the proposition at the arbitrary point $z\in Y$.
\end{proof}

\begin{prop}\label{nlines}
Let $(U,\phi\colon U\rightarrow\mathbb{R}^n)$ be an $n$-dimensional chart in a complete doubling Lipschitz differentiability space. When they exist, let $\gamma^x_1, \dots, \gamma^x_n$ be the $n$ ``broken curves'' through $x$ provided by Theorem \ref{Bate}, and let $v^x_i=(\phi\circ\gamma^x_i)'(0)$, which are $n$ linearly independent vectors in $\mathbb{R}^n$.

Then for almost every $x\in U$, every $(Y,y,G)\in \Tan(X,x,\phi)$ has the following property: 
There exists $L\geq 1$ such that, for every $z\in Y$, there are $n$ $L$-bi-Lipschitz lines $l_1, \dots, l_n$ with $l_i(0) = z$ that satisfy
$$ G(l_i(t)) = G(z) + tv^x_i $$
for all $t\in\mathbb{R}$.

The constant $L$ depends on the point $x$ but not on the sequence of scales defining the tangent.
\end{prop}
\begin{proof}
This follows immediately from the previous two results.
\end{proof}

\section{Obtaining Lipschitz quotient maps}\label{LQsection}

Recall that a \textit{Lipschitz quotient map} $f$ between metric spaces $X$ and $Y$ is a Lipschitz map for which there exists $c>0$ such that
$$f(B(x,r)) \supseteq B(f(x),cr)$$
for any ball $B(x,r)$ in $X$. The constant $c$ is called the \textit{co-Lipschitz} constant of the map.

A simple reformulation of the Lipschitz quotient property is the following. A Lipschitz map $f\colon X\rightarrow Y$ is a Lipschitz quotient map with co-Lipschitz constant $c$ if and only if
$$ \dist_X(x,f^{-1}(y)) \leq c^{-1} d_Y(f(x),y)$$
for every $x\in X$ and $y\in Y$.

\begin{cor}\label{lipquot}
Let $(U,\phi\colon U\rightarrow\mathbb{R}^n)$ be an $n$-dimensional chart in a complete doubling Lipschitz differentiability space $X$. Then for almost every $x\in U$, every $(Y,y,F)\in \Tan(X,x,\phi)$ has the property that $F$ is a Lipschitz quotient map onto $\mathbb{R}^n$.

The Lipschitz and co-Lipschitz constants associated to the Lipschitz quotient map $F$ depend on the point $x$, but not on the sequence of scales defining the tangent.
\end{cor}
\begin{proof}
We may assume that $x$ lies in the full measure set provided by Proposition \ref{nlines}. As in Proposition \ref{nlines}, we have $n$ ``broken curves'' $\gamma^x_i$ through $x$, from Theorem \ref{Bate}. Let $v_i=(\phi\circ\gamma^x_i)'(0)$, which are $n$ linearly independent vectors in $\mathbb{R}^n$. To simplify the proof, we first fix a linear map $A\colon \mathbb{R}^n\rightarrow\mathbb{R}^n$ that sends each $v_i$ to $e_i$, the $i$th standard basis vector of $\mathbb{R}^n$. Note that $A$ is invertible, because $\{v_i\}$ is a linearly independent set.

Now let $\psi = A\circ \phi$. It is clear that $(U,\psi)$ is still an $n$-dimensional chart, so we can apply Proposition \ref{nlines} to obtain $L\geq 1$ and $(Y,y,G)\in \Tan(X,x,\psi)$ with the property that for every $z\in Y$, there are $n$ $L$-bi-Lipschitz lines $l^z_1, \dots, l^z_n$ with $l^z_i(0) = z$ that satisfy
$$ G(l^z_i(t)) = G(z) + te_i $$
for all $t\in \mathbb{R}$.

We now show that $G$ is a Lipschitz quotient map. As a tangent of a Lipschitz map, it is automatically Lipschitz. To establish the co-Lipschitz bound, it suffices (by the remark above) to show that there is a constant $C>0$ such that, whenever $z\in Y$ and $p\in\mathbb{R}^n$,
\begin{equation}\label{LQequivalent}
\text{dist}(z, G^{-1}(p)) \leq C |G(z) - p|.
\end{equation}

Fix $z\in Y$ and $p\in\mathbb{R}^n$. Let $q=G(z)\in\mathbb{R}^n$. Write $p = (p_1, p_2, \dots, p_n)\in\mathbb{R}^n$ and $q=(q_1, q_2, \dots, q_n) \in \mathbb{R}^n$.

Let $z_1 = l^z_1(p_1 - q_1)$. Then
$$ G(z_1) = q+ (p_1-q_1)e_1 = (p_1, q_2, \dots, q_n).$$

Let $z_2 = l^{z_1}_2(p_2 - q_2)$. Then
$$ G(z_2) = G(z_1) + (p_2-q_2)e_2 = (p_1, p_2, q_3, \dots, q_n).$$

Repeating this $n$ times, we obtain $z_n$ such that $G(z_n) = p$. In addition, 
\begin{align*}
d_Y(z_n, z) &\leq d(z,z_1) + d(z_1, z_2) + \dots + d(z_{n-1}, z_n)\\
&\leq L|p_1 - q_1| + L|p_2 - q_2| + \dots + L|p_n - q_n|\\
&\leq Ln^{1/2}|G(z) - p|.
\end{align*}

Because $z_n\in G^{-1}(p)$, this proves (\ref{LQequivalent}) and so concludes the proof that $G$ is a Lipschitz quotient map with co-Lipschitz constant $c=(Ln^{1/2})^{-1}$. Now consider the space-function $(Y,y,F)\in\Tan(X,x,\phi)$ associated to the same sequence of scales as $(Y,y,G)\in\Tan(X,x,\psi)$. As $A\circ \phi = \psi$ and taking tangents is a linear operation on functions, we see that $A\circ F = G$. Therefore $F = A^{-1} \circ G$, and since $A$ is bi-Lipschitz, $F$ is also a Lipschitz quotient map.

The bi-Lipschitz constant of $A$ depends only on the vectors $\{v_i\}$ and not on the sequence of scales defining the tangent. Therefore, the Lipschitz and co-Lipschitz constants of $F$ also do not depend on the sequence of scales defining the tangent.
\end{proof}

The following corollary summarizes two simple immediate consequences.
\begin{cor}\label{lipquotfacts}
Let $(U,\phi)$ be an $n$-dimensional chart in a complete doubling Lipschitz differentiability space $X$. Then the following two facts hold:
\begin{enumerate}[\normalfont (i)]
\item At almost every point of $U$, any tangent of $\phi$ maps onto $\mathbb{R}^n$. \label{surjtangent}
\item\label{lowermass}  For almost every point $x\in U$, there is a constant $c_0>0$ such that any tangent $(Y,y)\in \Tan(X,x)$ satisfies the lower mass bound
$$ \mathcal{H}^n(B(z,r)) \geq c_0 r^n $$
for all $z\in Y$ and $r>0$.
\end{enumerate}
\end{cor}
Corollary \ref{lipquotfacts} is an analog of Theorem 13.4 of \cite{Ch99} from the setting of PI spaces.

\section{Uniformly rectifiable tangents}\label{recttangent}
The proof of Theorem \ref{Rntangent} will be an immediate application of Corollary \ref{lipquot} and the main theorem of \cite{Da88}, which we discuss briefly.

\subsection{David's bi-Lipschitz pieces result}

Let $(X,d)$ be a complete Ahlfors $n$-regular metric space. Such a space admits a standard ``dyadic cube decomposition''. We describe this briefly, although the exact details are not really important for us. 

The formulation in \cite{Se00}, Section 2.3, is the easiest to apply here. It says that there exists $j_0\in\mathbb{Z}\cup\{\infty\}$ (with $2^{j_0} \leq \diam X < 2^{j_0+1}$ if $X$ is bounded) such that for each $j<j_0$, there exists a partition $\Delta_j$ of $X$ into measurable subsets $Q\in \Delta_j$ such that

\begin{itemize}\label{cubes}
\item $Q\cap Q' = \emptyset$ if $Q,Q'\in \Delta_j$ and $Q\neq Q'$.
\item If $j\leq k < j_0$ and $Q\in \Delta_j, Q'\in \Delta_k$, then either $Q\subseteq Q'$ or $Q\cap Q' = \emptyset$. 
\item $C_0^{-1} 2^j \leq \diam Q \leq C_0 2^j$ and $C_0^{-1} 2^{nj} \leq \haus{Q} \leq C_0 2^{nj}$.
\item For every $j< j_0$ and $Q\in \Delta_j$, there is a point $x\in Q$ such that $B(x, c_0 2^j) \subseteq Q \subseteq B(x, C_0 2^j)$.
\end{itemize}

The constants $c_0$ and $C_0$ in the cube decomposition depend only on $n$ and the Ahlfors-regularity constant of the space.

Note that these dyadic cubes are not necessarily closed or open, but merely measurable. They are also disjoint, and do not merely have disjoint interiors.

The following condition was introduced by David \cite{Da88}. Although we state the full, rather technical, condition, we will only use a very simple consequence of it, Corollary \ref{lqbilip} below.

\begin{definition}\label{davidscondition}
Let $(X,d)$ be a complete Ahlfors $n$-regular metric space with a system of dyadic cubes as above. Let $I_0$ be a cube in $X$, and $z\colon I_0\rightarrow \mathbb{R}^n$ be a Lipschitz map. We will say that $z$ satisfies \textit{David's condition} on $I_0$ if the following holds:

For every $\lambda, \gamma > 0$, there exist $\Lambda, \eta > 0$ such that, for every $x\in I_0$ and $j<j_0$, if $T$ is the union of all $j$-cubes in $X$ intersecting $B(x,\Lambda 2^j)$, and if $T\subseteq I_0$ and $|z(T)|\geq \gamma |T|$, then either:
\begin{enumerate}[(i)]
\item \label{containsball} $z(T) \supseteq B(z(x), \lambda 2^{j})$, or 
\item \label{expandscube} there is a $j$-cube $R\subset T$ such that
$$ |z(R)|/|R| \geq (1+2\eta)|z(T)|/|T|$$
\end{enumerate}
\end{definition}

Theorem 1 of \cite{Da88} (see also the more general Theorem 10.1 of \cite{Se00}) says the following:
\begin{thm}[\cite{Da88}, Theorem 1]\label{davidtheorem}
Let $(X,d)$ be a complete Ahlfors $n$-regular metric space with a system of dyadic cubes as above. Let $I_0$ be a cube in $X$, and $z\colon I_0\rightarrow\mathbb{R}^n$ be a Lipschitz map. Suppose that $|z(I_0)|\geq \delta |I_0|$ for some $\delta>0$, and that $z$ satisfies David's condition on $I_0$.

Then there is a set $E\subset I_0$ with $|E|\geq \theta$ such that $|z(x)-z(y)|\geq M^{-1} d(x,y)$ for all $x,y\in E$.

The constants $\theta\geq 0$ and $M\geq 0$ depend only on the constants associated to the space, the Lipschitz constant of $z$, and the numbers $\delta$, $\Lambda$ and $\eta$ (for $\lambda=1$ and $\gamma=\delta/2$).
\end{thm}

We now note that if $(X,d)$ is Ahlfors $n$-regular and $z\colon X\rightarrow\mathbb{R}^n$ is a Lipschitz quotient map, then $z$ trivially satisfies David's condition on any cube. Indeed, suppose $z$ is a Lipschitz quotient map, so that $z(B(x,r))\supseteq B(z(x),cr)$ for all $x\in X$ and $r>0$. Fix a cube $I_0\subseteq X$ and constants $\lambda, \gamma>0$. Set $\Lambda = \lambda/c$ and $\eta$ arbitrary. Let $T$ be the union of all $j$-cubes in $X$ touching $B=B(x,\Lambda 2^j)$, where $x\in I_0$. If $T\subseteq I_0$, then $B\subseteq T \subseteq I_0$. We therefore see that
$$ z(T) \supseteq z(B) \supseteq B(z(x), c\Lambda 2^j) = B(z(x), \lambda 2^j). $$
Thus the first branch of David's condition is always satisfied.

The following is therefore an immediate consequence of David's Theorem \ref{davidtheorem}, which we record so that we can reference it later.

\begin{cor}\label{lqbilip}
Let $(Y,d)$ be a complete Ahlfors $n$-regular and let $f\colon Y\rightarrow\mathbb{R}^n$ be a Lipschitz quotient map. Then, for every ball $B$, $f$ is $\beta$-bi-Lipschitz on some subset of $B$ of measure at least $\alpha\haus{B}$. Here $\alpha,\beta>0$ depend only on the Ahlfors regularity constant of the space and the Lipschitz and co-Lipschitz constants of $f$.

In particular, $Y$ is uniformly rectifiable, with constants $\alpha$ and $\beta$.
\end{cor}

\subsection{Proof of Theorem \ref{Rntangent} and Corollary \ref{banachtangent}}
We now apply Corollaries \ref{lipquot} and \ref{lqbilip} to prove Theorem \ref{Rntangent}. Let $X$ be an Ahlfors $n$-regular Lipschitz differentiability space containing a chart $(U,\phi \colon U\rightarrow\mathbb{R}^n)$ of dimension $n$. Note that, as mentioned above, any tangent $Y$ of $X$ is Ahlfors $n$-regular.

By Corollary \ref{lipquot}, for almost every point $x$ of $U$, there exists $(Y,y)\in \Tan(X,x)$ and a Lipschitz quotient map $G\colon Y\rightarrow\mathbb{R}^n$. It follows immediately from Corollary \ref{lqbilip} that $Y$ is uniformly rectifiable.

For the second part of the theorem, take a positive measure subset $E$ of $Y$ on which $G$ is a bi-Lipschitz map. Fix a point of density $y'$ of this set such that $G(y')$ is a point of $\mathcal{H}^n$-density of $G(E) \subset \mathbb{R}^n$. Take a further tangent $(Z, z, H)\in \Tan(Y, y', G)$. Note that $(Z,z)\in\Tan(Y,y')\subset\Tan(X,x)$ (see \cite{Le11}, Theorem 1.1). In addition, it follows from Lemma \ref{densitytangent} and a standard argument that $H$ is a bi-Lipschitz map from $Z$ onto $\mathbb{R}^n$. To see this, we may use the type of convergence discussed in \cite{DS97}, Chapter 8, to simultaneously blow up $G^{-1}\colon G(E)\rightarrow Y$ at $G(y')$, yielding a bi-Lipschitz map $\overline{H}\colon \mathbb{R}^n \rightarrow Z$ such that $H(\overline{H}(w))=w$ for all $w\in \mathbb{R}^n$. This shows that $H$ is surjective and therefore that $(Z,z)\in \Tan(X,x)$ is bi-Lipschitz equivalent to $\mathbb{R}^n$. This completes the proof of Theorem \ref{Rntangent}.

Corollary \ref{banachtangent} is an immediate consequence of Theorem \ref{Rntangent} and Theorem 9 of \cite{Ki94}, again using the fact \cite{Le11} that, at almost every point of $X$, tangents of tangents are tangents.

\section{Proof of Theorem \ref{strongunrectthm}}\label{unrect}
We now consider an Ahlfors $s$-regular Lipschitz differentiability space $X$ with a $k$-dimensional chart $U$, where $k<s$.

Fix any $N\in\mathbb{N}$ and any Lipschitz function $f\colon U\rightarrow\mathbb{R}^N$. We will show that $\mathcal{H}^s(f(U))=0$.

Without loss of generality, we may assume that $f$ is $1$-Lipschitz, $N\geq s$, and $U$ is bounded. Write $f = (f_1, \dots, f_N)$, where $f_i\colon X\rightarrow \mathbb{R}$ for $1\leq i \leq N$. We say that $f$ is differentiable at $x\in U$ if each $f_i$ is differentiable at $x$. In this case, we write $Df_x$ for the $N\times k$ matrix whose $i$th row is $d(f_i)_x \in \mathbb{R}^k$.

Note that the subset of $U$ on which $f$ is non-differentiable has $\mathcal{H}^s$-measure zero, and thus so does its image under $f$. It therefore suffices to show that $\mathcal{H}^s(f(V))=0$, where $V\subseteq U$ is the subset on which $f$ is differentiable. To do so, it suffices to show that the Hausdorff content $\mathcal{H}^s_\infty$ of $f(V)$ is zero (see \cite{He01}, p. 61).

Fix $\delta>0$. For each $x\in V$, choose $r_x\in(0,1)$ small so that
$$ y\in B(x,6r_x) \Rightarrow |f(y)-f(x)-Df_x\cdot (\phi(y) - \phi(x))| < \delta r_x. $$

By a basic covering theorem, (see \cite{He01}, Theorem 1.2), we may acquire a collection of balls $\{B_j = B(x_j, r_j)\}$, with $x_j\in V$ and $r_j=5r_{x_j}$, covering $V$ such that the collection $\{\frac{1}{5} B_j\}$ consists of pairwise disjoint sets.

Let $P_j$ denote the $k$-dimensional affine space $f(x_j)+Df_{x_j}[\mathbb{R}^k]\subset \mathbb{R}^N$. Then
$$f(B_j) \subset N_{\delta r_j}(B(f(x_j), r_j) \cap P_j). $$ 

Thus, $f(B_j)$ can be covered by $C\delta^{-k}$ balls of radius $2\delta r_j$, where $C$ depends only on $k$ (cover the $k$-dimensional Euclidean ball $B(f(x_j), r_j)\cap P_j$ by balls of radius $\delta r_j$ and then double the radii of these balls).

Note that because $V$ is bounded and the balls $B(x_j, \frac{1}{5} r_j)$ are disjoint, we have that
\begin{equation}\label{rjsum}
\sum_j r_j^s = 5^s \sum_j (r_j/5)^s \leq C_0 5^s \mathcal{H}^s(N_1(V)) < \infty,
\end{equation}
where the first inequality is because the collection $\{B(x_j, r_j/5)\}$ consists of disjoint subsets of $N_1(V)$, and the second inequality is because $V$ is bounded and $X$ is Ahlfors $s$-regular with constant $C_0$.

Thus,
\begin{align*}
\mathcal{H}^s_\infty(f(V)) &\leq C\sum_j \delta^{-k} (2\delta r_j)^s\\
&\leq 2^sC\delta^{s-k} \sum_j r_j^s\\
&\leq 10^sC C_0 \mathcal{H}^s(N_1(V)) \delta^{s-k},
\end{align*}
using (\ref{rjsum}).

Because $s-k>0$ and $\mathcal{H}^s(N_1(V))<\infty$, sending $\delta\rightarrow 0$ completes the proof of Theorem \ref{strongunrectthm}.

\begin{rmk}\label{manifoldcharts}
In \cite{Se96}, Semmes shows that a linearly locally contractible, Ahlfors $n$-regular $n$-manifold $M$ admits a Poincar\'e inequality, and is therefore a Lipschitz differentiability space. Using Theorem \ref{strongunrectthm} above, combined with the deep Theorem 1.29 of \cite{Se96}, one can give a straightforward proof that the differentiable structure of $M$ consists of $n$-dimensional charts. This is done in the following way:

The fact that the charts in $M$ have dimension at most $n$ follows from Theorem 13.8 of \cite{Ch99}, or alternatively from Corollary \ref{assouad} below.

If a chart $U$ in $M$ had dimension $k<n$, then $U$ would be strongly unrectifiable by Theorem \ref{strongunrectthm}. Fixing a point of density $x$ of $U$, Theorem 1.29 of \cite{Se96} provides, for all $j\in\mathbb{N}$, mappings $f_j\colon B(x,j^{-1})\rightarrow \mathbb{S}^n$ that are $Cj$-Lipschitz, for some constant $C$, and whose images have full measure in the standard unit sphere $\mathbb{S}^n$. In other words, we have
$$ \mathcal{H}^n(f_j(B(x,j^{-1}))) = \mathcal{H}^n(\mathbb{S}^n) $$
independent of $j$.
On the other hand, by the strong unrectifiability of $U$ we have
$$ \mathcal{H}^n(f_j(B(x,j^{-1}) \cap U)) = 0. $$
Letting $j$ tend to infinity, one easily obtains a contradiction to the assumption that $x$ is a point of density of $U$.
\end{rmk}

\begin{rmk}
In fact, as was shown by the author in \cite{GCD13}, the spaces described in Remark \ref{manifoldcharts} are locally uniformly rectifiable in dimension $n$, which is much stronger than having an $n$-dimensional differentiable structure.
\end{rmk}

\section{Additional corollaries}\label{additionalcors}
This section contains some further results that follow from Corollary \ref{lipquotfacts} and Theorem \ref{Rntangent}.

\subsection{Embedding and rectifiability}
The fact that blowups of the coordinate functions are surjective (statement (\ref{surjtangent}) in Corollary \ref{lipquotfacts}) appears to be new for Lipschitz differentiability spaces (as opposed to PI spaces, where it appears in Theorem 13.4 of \cite{Ch99}). In this section, we give some consequences of this fact. (While this paper was in preparation, we learned of Schioppa's paper \cite{Sc13}, in Section 5 of which he also proves that the blowups of the coordinate functions are surjective in a Lipschitz differentiability space. The results in this subsection, which are all corollaries of that fact, can thus also be derived from Schioppa's work.)

Exactly as for PI spaces (see \cite{Ch99}, Theorems 14.1 and 14.2), surjectivity of blowups gives the following consequences.

\begin{cor}\label{embedding} 
Let $(X,d,\mu)$ be a complete Lipschitz differentiability space with an $n$-dimensional chart $(U,\phi)$. Suppose that $F\colon X\rightarrow\mathbb{R}^N$ is a bi-Lipschitz embedding. Then for almost every $x\in U$, the set $F(X)$ has a unique tangent at $F(x)$ that is an $n$-dimensional linear subspace of $\mathbb{R}^N$.

If in addition $\mathcal{H}^n(U)<\infty$ and $\mathcal{H}^n$ is absolutely continuous with respect to $\mu$, then it follows that $F(U)$, and therefore $U$, is $n$-rectifiable.
\end{cor}
\begin{proof}
The proof proceeds as for PI spaces. Consider any point of density $x$ of $U$ at which $F$ is differentiable; we may also assume that $F(x)$ is a point of $F_*(\mu)$-density of $F(U)$. Note that the push-forward measure $F_*(\mu)$ is doubling on $F(X)$, as $F$ is bi-Lipschitz.

Take a tangent $Z\subset\mathbb{R}^N$ of $F(U)$ at $F(x)$ along some sequence of scales. Note that this tangent $Z$ can be realized as a subset of $\mathbb{R}^N$, as in \cite{DS97}, Lemma 8.2.

Simultaneously blow up $X$ and the maps $\phi$ and $F$ at $x$ to obtain a tangent $Y$ of $X$ at $x$, a Lipschitz map $\tilde{\phi}\colon Y\rightarrow\mathbb{R}^n$, and a bi-Lipschitz map $\tilde{F}\colon Y\rightarrow Z\subset\mathbb{R}^N$. To summarize, we obtain, along some fixed sequence of scales,
$$ (Z,z)\in \Tan(F(U), F(x)) = \Tan(F(X),F(x)),$$
$$ (Y,y,\tilde{F})\in \Tan(X,x,F), \text{ and }$$
$$ (Y,y,\tilde{\phi}) \in \Tan(X,x,\phi),$$
where $\tilde{F}$ is a bi-Lipschitz map of $Y$ into $Z$. (The fact noted above that $\Tan(F(U), F(x)) = \Tan(F(X),F(x))$ follows from Lemma \ref{densitytangent} and the fact that $F_*(\mu)$ is doubling.)

In fact, by the same standard argument used in the proof of Theorem \ref{Rntangent}, $\tilde{F}$ maps $Y$ onto $Z$. Recall that we show this by using the type of convergence discussed in \cite{DS97}, Chapter 8, to simultaneously blow up $F^{-1}\colon F(U)\rightarrow U$ at $F(x)$, yielding a bi-Lipschitz map $G\colon Z\rightarrow Y$ such that $\tilde{F}(G(w))=w$ for all $w\in Z$. This shows that $\tilde{F}$ is surjective.

Now, because $F$ is differentiable at $x$, its blowup $\tilde{F}$ can be written as
$$\tilde{F} = DF_x \circ \tilde{\phi},$$
where $DF_x$ is a linear map $\mathbb{R}^n\rightarrow\mathbb{R}^N$ and $\tilde{\phi}$ is the blowup of $\phi$. Note that $DF_x$ must have full rank $n$, because $\tilde{F}$ is a bi-Lipschitz map.

Because $\tilde{\phi}$ is surjective (see Corollary \ref{lipquotfacts}), we have
$$ Z = F(Y) = DF_x(\tilde{\phi}(Y)) = DF_x(\mathbb{R}^n), $$
which is a fixed $n$-dimensional linear subspace of $\mathbb{R}^N$ independent of the choice of scales defining the tangent $Z$. This shows that for $\mu$-almost every $x\in U$, $F(U)$ has a unique tangent at $F(x)$ that is an $n$-dimensional linear subspace of $\mathbb{R}^N$.

The second part of the corollary, about the rectifiability of $U$, now follows from a well-known characterization of rectifiable sets in Euclidean space (see \cite{Ma95}, Theorem 15.19).
\end{proof}

\begin{rmk}\label{CsornyeiJones}
In the second part of Corollary \ref{embedding}, note that if $\mathcal{H}^n(U)=0$ then the $n$-rectifiability of $U$ holds for trivial reasons. However, if one appeals to a result recently announced by Cs\"ornyei and Jones, it is possible to show that an $n$-dimensional chart $U$ always has $\mathcal{H}^n(U)>0$ (see \cite{Ba13}, Remark 6.11).
\end{rmk}

Note that the proof of Corollary \ref{embedding} shows that, if $X$ is a Lipschitz differentiability space with an $n$-dimensional chart $(U,\phi)$, and if $F\colon X\rightarrow \mathbb{R}^N$ is bi-Lipschitz, then at almost every point $x\in U$, every tangent of $X$ at $x$ is bi-Lipschitz equivalent to $\mathbb{R}^n$. The following non-embedding result is therefore an immediate consequence.

\begin{cor}\label{nonembedding}
Let $X$ be a complete Lipschitz differentiability space with an $n$-dimensional chart $(U,\phi)$. Suppose there exists a set $A \subseteq U$ of positive measure such that for every $a\in A$, there exists $(Y,y)\in \Tan(X,a)$ that is \textit{not} bi-Lipschitz equivalent to $\mathbb{R}^n$. Then $X$ does not admit a bi-Lipschitz embedding into any Euclidean space.
\end{cor}

The previous two results greatly restrict the subsets of Euclidean space that can admit differentiable structures. For example, we obtain the following non-existence result. Here $|\cdot|$ refers to the standard Euclidean metric.

\begin{cor}\label{Eucsubsets}
Let $E$ be a closed subset of some $\mathbb{R}^N$ that is Ahlfors $s$-regular, where $0<s\leq N$. If $s$ is not an integer, then $(E, |\cdot|, \mathcal{H}^s)$ is not a Lipschitz differentiability space.
\end{cor}
\begin{proof}
Suppose that $Q$ is not an integer but that $(E, |\cdot|, \mathcal{H}^Q)$ is in fact a Lipschitz differentiability space. Because $E$ is Ahlfors $s$-regular, so are all its tangents. On the other hand, by Corollary \ref{embedding}, some tangent of $E$ must be a linear subspace of $\mathbb{R}^N$, and so must have integer Hausdorff dimension. This is a contradiction.
\end{proof}

In particular, many self-similar fractals like the standard Sierpi\'nski carpet and Sierpi\'nski gasket cannot be Lipschitz differentiability spaces when equipped with their canonical measures. In general, Ahlfors regular spaces with non-integer Hausdorff dimension can be Lipschitz differentiability spaces and can even admit Poincar\'e inequalities (see \cite{BP99}, \cite{La00}, \cite{CK13_PI}). Such spaces can never admit bi-Lipschitz embeddings into any Euclidean space. Indeed, in the case of PI spaces, stronger non-embedding results hold (see \cite{CK09}).

Surjectivity of the blowups also implies a sharp bound on the dimension of a differentiable structure on a doubling space. This uses the notion of the Assouad dimension $\dim_A X$ of a metric space $X$; a definition can be found in \cite{He01}, Definition 10.15.

\begin{cor}\label{assouad}
Let $X$ be a doubling Lipschitz differentiability space with an $n$-dimensional chart $(U,\phi)$. Then $n\leq \dim_A X$.
\end{cor}
\begin{proof}
This follows from two facts about Assouad dimension. First, the Assouad dimension of a space $X$ is always at least the Hausdorff dimension $\dim_H X$ of $X$ (see \cite{MT10}, Section 1.4.4). Second, the Assouad dimension of a tangent space is always at most the Assouad dimension of the original space (\cite{MT10}, Proposition 6.1.5).

Because in addition the blowups of the coordinates yield a Lipschitz map from a tangent $Y$ of $X$ onto $\mathbb{R}^n$, we have that
$$ \dim_A X \geq \dim_A Y \geq \dim_H Y \geq n. $$
\end{proof}

Note that, for example, the Assouad dimension of $\mathbb{R}^n$ is the same as the dimension of its differentiability charts, so Corollary \ref{assouad} is sharp. Corollary \ref{assouad} was first noted by Schioppa in \cite{Sc13}, Section 5.

\subsection{Spaces with quasi-M\"obius symmetries}\label{qmsymmetry}

In \cite{BK02_rigidity}, Bonk and Kleiner consider compact metric spaces that admit the following type of symmetries. For the definition of quasi-M\"obius maps, see \cite{BK02_rigidity}.

\begin{definition}\label{qms}
A compact metric space $X$ \textit{admits quasi-M\"obius symmetries} if the following holds: every triple of points in $X$ can be blown up to a uniformly separated triple by a uniformly quasi-M\"obius map. In other words, there is a homeomorphism $\eta\colon [0,\infty)\rightarrow[0,\infty)$ and a constant $\delta>0$ such that for every triple of points $x,y,z\in X$, there is a $\eta$-quasi-M\"obius map $g \colon X\rightarrow X$ such that the points $g(x)$, $g(y)$, $g(z)$ have mutual distance at least $\delta$.
\end{definition}

This condition is satisfied, for example, by the boundaries of hyperbolic groups equipped with their visual metrics.

Bonk and Kleiner show the following theorem. 
\begin{thm}[\cite{BK02_rigidity}, Theorem 6.1]
If a compact Ahlfors $n$-regular metric space $X$ admits quasi-M\"obius symmetries and in addition has topological dimension $n$, then $X$ is quasi-M\"obius equivalent to the standard sphere $\mathbb{S}^n$.
\end{thm}

In other words, if a space admits quasi-M\"obius symmetries and has extremal topological dimension, then it must be the standard sphere.

An immediate consequence of Theorem \ref{Rntangent} and the methods of \cite{BK02_rigidity} is the following alternate version of Bonk and Kleiner's result, in which the assumption of extremal topological dimension is replaced by the assumption of extremal ``differentiability dimension'':

\begin{cor}
Let $X$ be a compact Ahlfors $n$-regular Lipschitz differentiability space containing a chart $U$ of dimension $n$. Suppose that $X$ admits quasi-M\"obius symmetries, as in Definition \ref{qms}. Then $X$ is quasi-M\"obius equivalent to $\mathbb{S}^n$.
\end{cor}
\begin{proof}
By Theorem \ref{Rntangent}, $X$ admits a tangent $Y$ that is bi-Lipschitz equivalent to $\mathbb{R}^n$. It follows from Lemma 5.8 of \cite{BK02_rigidity} (see also the remark in the proof of Theorem 6.1 of that paper) that $X$ is quasi-M\"obius equivalent to $\mathbb{S}^n$.
\end{proof}

The assumption that $X$ is a Lipschitz differentiability space is strong, but it is somewhat natural in this context: In \cite{BK05}, Bonk and Kleiner show that if an Ahlfors regular space admits quasi-M\"obius symmetries with no common fixed point and in addition is extremal for conformal dimension, then it supports a Poincar\'e inequality and is therefore a Lipschitz differentiability space.

\bibliography{unifiedbib}{}
\bibliographystyle{plain}

\end{document}